\documentclass[12pt]{amsart}
\usepackage{amssymb,amsxtra,amsmath,amsfonts}
\usepackage{mleftright}

\copyrightinfo{}{}

\theoremstyle{plain}
\newtheorem{theorem}{Theorem}
\newtheorem{lemma}[theorem]{Lemma}
\newtheorem{proposition}[theorem]{Proposition}

\theoremstyle{definition}

\theoremstyle{remark}
\newtheorem{remark}[theorem]{Remark}

\numberwithin{equation}{section}
\numberwithin{theorem}{section}

\newcommand\seref{Section \ref}

\newcommand{\lie}[1]{\mathfrak{#1}}
\newcommand{\M}[2]{\renewcommand\arraystretch{1.3}\mleft[\begin{array}{#1} #2 \end{array}\mright]}
\newcommand{\ii}{\mathrm{i}} 
\newcommand{\CC}{\mathbb{C}}
\newcommand{\B}{\mathrm{B}}
\newcommand{\ph}{\varphi}
\newcommand{\ka}{\kappa}
\newcommand{\al}{\alpha}
\newcommand{\ga}{\gamma}
\newcommand{\la}{\lambda}
\newcommand{\Ve}{V_{\bar{0}}}
\newcommand{\Vo}{V_{\bar{1}}}
\newcommand{\Id}{\mathrm{Id}}

\DeclareMathOperator\Sp{\mathrm{Sp}}
\DeclareMathOperator\GL{\mathrm{GL}}
\DeclareMathOperator\Or{\mathrm{O}}
\DeclareMathOperator\re{Re}
\DeclareMathOperator\im{Im}
\DeclareMathOperator\Span{span}
\DeclareMathOperator\diag{diag}

\begin{document}

\title{Inhomogeneous supersymmetric bilinear forms}

\author{Bojko Bakalov}
\address{Department of Mathematics, North Carolina State University, 
Raleigh, NC 27695, USA}
\email{bojko\_bakalov@ncsu.edu}

\author{McKay Sullivan}
\address{Department of Mathematics, Dixie State University, Saint George, UT 84770, USA}
\email{mckay.sullivan@dixie.edu}


\date{December 30, 2016; Revised September 20, 2017}

\begin{abstract}
We consider inhomogeneous supersymmetric bilinear forms, i.e., forms
that are neither even nor odd. We classify such forms up to dimension
seven in the case when the restrictions of the form to the even and
odd parts of the superspace are nondegenerate. As an application, we
introduce a new type of oscillator Lie superalgebra.
\end{abstract}

\thanks{The first author is supported in part by a Simons Foundation grant 279074} 


 \maketitle


\section{Introduction}\label{sec:intro}

Many important Lie superalgebras can be constructed as subalgebras of oscillator superalgebras (see, e.g., \cite{FAP}). For other good references on Lie superalgebras see \cite{K2}, \cite{CW}, and \cite{M}.  To obtain an oscillator Lie superalgebra one may start with a superspace $V = \Ve\oplus \Vo$ over $\CC$ and a \emph{skew-supersymmetric} bilinear form $(\cdot | \cdot)\colon V \times V \to V$, i.e., such that 
\begin{equation*}
(a|b)=-(-1)^{p(a)p(b)} (b|a)
\end{equation*}
for all homogeneous $a,b\in V$ of parities $p(a)$ and $p(b)$, respectively.
One usually takes this form to be nondegenerate and \emph{even}, i.e., $(\Ve|\Vo) = 0$. These two conditions imply that the restrictions of the form to $\Ve$ and $\Vo$ are also nondegenerate. Then a Lie superalgebra structure is given to the extension $A = V\oplus \CC K$  by declaring $K$ to be an even central element and letting
\begin{equation*}
[a,b] = (a|b)K, \qquad a,b\in V.
\end{equation*}
The universal enveloping algebra $\mathcal{U}(A)$ is known as
the \emph{oscillator algebra}. It admits highest weight representations, which can be restricted to give representations of its subalgebras.

We will call a bilinear form $(\cdot | \cdot)$ on a vector superspace $V$ \emph{inhomogeneous} if $(\Ve|\Vo) \ne 0$ and $(V_i|V_i) \ne 0$ for some $i\in\{\bar0,\bar1\}$.
Our original motivation for studying inhomogeneous forms was to investigate representations of subalgebras of  the oscillator algebras obtained by relaxing the assumption that the bracket of an even (bosonic) oscillator and an odd (fermionic) oscillator must be zero. 
Though skew-supersymmetric forms are used to construct oscillator algebras, there is a natural one-to-one correspondence between supersymmetric forms on $V$ and skew-supersymmetric forms on $V^\Pi$, where $V^\Pi$ is the superspace obtained by reversing the parity of the elements of $V$:
\begin{equation*}
\Ve^\Pi = \Vo, \qquad \Vo^\Pi = \Ve.
\end{equation*}
 The inhomogeneous supersymmetric bilinear forms whose restrictions to $\Ve$ and $\Vo$ are nondegenerate will be called \textit{pre-oscillator forms}, and the algebras obtained as universal enveloping algebras of central extensions of $V^\Pi$ \textit{inhomogeneous oscillator algebras}.

Throughout the rest of the paper, every bilinear form $(\cdot|\cdot)\colon V \times V \to \CC$ will be assumed to be a pre-oscillator form. Then we can choose a homogeneous basis for $V$ so that the Gram matrix of the form is a block matrix of type
\begin{equation}\label{Gram}
G=\M{c|c}{
I_k & B \\
\hline
B^T & J_{2\ell}
},
\end{equation}
where $\dim \Ve=k$, $\dim \Vo = 2\ell$, and
\begin{equation*}
J_{2\ell} = \diag(J,\ldots,J), \qquad J= \M{cc}{ 0 & 1 \\ -1 & 0}.
\end{equation*}

\begin{remark}
Further relaxing the assumption that the restrictions of the form to even and odd parts are nondegenerate also leads to oscillator-like algebras. An example is the superspace spanned by an even vector $v_1$ and an odd vector $v_2$, with the form given by  $(v_1|v_1)=(v_1|v_2)=(v_2|v_1)=1$,  $(v_2 | v_2) = 0$.
 \end{remark}

A bilinear form $(\cdot | \cdot)$ on a superspace $V$ will be called \textit{reducible} if $V=U \oplus W$ is an orthogonal direct sum of subsuperspaces, i.e., if $(U | W)=0$. Otherwise $(\cdot|\cdot)$ is \textit{irreducible}. A natural first step in investigating inhomogeneous oscillator algebras is to classify all irreducible pre-oscillator forms on a given superspace $V$. We call two bilinear forms $(\cdot | \cdot)_1$ and $(\cdot | \cdot)_2$  \textit{equivalent} if there exists an even automorphism $\ph:V \to V$ satisfying 
\begin{equation}\label{equivalence}
(\ph u | \ph v )_1 = (u|v)_2 
\end{equation}
for all $u,v \in V$.

In \seref{sec:invariants} we introduce invariants that help us distinguish between equivalence classes of forms. Then in \seref{sec:class} we find representatives of equivalence classes of irreducible forms and use the invariants to prove that they are irreducible and distinct. This allows us to obtain a classification of pre-oscillator forms on superspaces of dimension up to $7$. Finally, in \seref{sec:oscillator} we introduce oscillator Lie superalgebras obtained from inhomogeneous bilinear forms and discuss a 3-dimensional example.

\section{Invariants}\label{sec:invariants}

We aim to find invariants that distinguish the equivalence classes of forms on a given superspace $V$. We assume $(\cdot|\cdot)_1$ and $(\cdot|\cdot)_2$ are equivalent forms on $V$, and choose bases for $V$ so that the Gram matrices $G_i$ $(i=1,2)$ take the form \eqref{Gram} with $B=B_i$.
If $M$ is the matrix of an even automorphism $\ph$ satisfying \eqref{equivalence}, then \eqref{equivalence} can be written as a matrix equation
\begin{equation}\label{cob}
M^T G_1 M = G_2.
\end{equation}
Since $\ph$ is even, $M$ is a block matrix of the form
\begin{equation*}
M = \M{c|c}{
X & 0 \\
\hline
0 & Y
}.
\end{equation*}
Then \eqref{cob} holds if and only if we have:
\begin{align*}
X \in \Or(k) &= \{X \in \GL(k) \colon X^TX = I_k\}, \\
Y \in \Sp(2\ell) &= \{Y \in \GL(2\ell) \colon Y^TJ_{2\ell}Y = J_{2\ell}\}, \\
B_2 &= X^T B_1 Y.
\end{align*}
 Hence finding the equivalence class of a bilinear form with Gram matrix \eqref{Gram} is equivalent to finding the orbit of $B$ under the right action of $\Or(k) \times \Sp(2\ell)$ on $\CC^{k \times 2\ell}$ defined by 
\begin{equation}\label{jointaction}
B \cdot (X,Y) = X^TBY.
\end{equation}

Given $X \in \Or(k)$ and $Y \in \Sp(2\ell)$, let $A = X^TBY$. Then $A^TA = Y^TB^TBY$ is independent of $X$. Thus any invariant of the right action of $\Sp(2\ell)$ on $\CC^{2\ell \times 2\ell}$ given by 
\begin{equation*}
C \cdot Y = Y^TCY
\end{equation*}
is an invariant of $B^TB$ under the action \eqref{jointaction}. Similarly $AJ_{2\ell}A^T = X^TBJ_{2\ell}B^TX$ is independent of $Y$. Thus any invariant of the right action of $\Or(k)$ on $\CC^{k\times k}$ given by
\begin{equation}\label{Xconj}
C \cdot X =  X^TCX
\end{equation}
 is an invariant of $BJ_{2\ell}B^T$ under the joint action  \eqref{jointaction}.

Let us consider the action \eqref{Xconj}. Since $\det(X) = \pm 1$, it is apparent that $\det(X^TCX)=\det(C)$. Thus the determinant of $BJ_{2\ell}B^T$ is an invariant of the action \eqref{jointaction}. More generally, we can write 
\begin{equation}
X^T(BJ_{2\ell}B^T-\la I_k)X = (X^TBY)J_{2\ell}(X^TBY)^T - \la I_k.
\end{equation}
This implies that the characteristic polynomials of the two matrices $BJ_{2\ell}B^T$ and $(X^TBY)J_{2\ell}(X^TBY)^T$ are equal. 
Therefore, the characteristic polynomial 
\begin{equation*}
P_{B}(\la) = \det(BJ_{2\ell}B^T-\la I_k)
\end{equation*} is invariant under the joint action \eqref{jointaction}. 
We have obtained the following theorem.

\begin{theorem}\label{polynomials}
Let\/ $V=\Ve \oplus \Vo$ be a superspace with\/ $\dim \Ve = k$ and\/ $\dim \Vo = 2\ell$, and let\/ $(\cdot | \cdot) \colon V \times V \to V$ be an inhomogeneous nondegenerate supersymmetric bilinear form on\/ $V$ with Gram matrix\/ \eqref{Gram} in a given homogeneous basis. Then each coefficent of the characteristic polynomial\/ $P_{B}(\la)$ is an invariant for the joint action\/ \eqref{jointaction} of\/ $\Or(k) \times \Sp(2\ell)$ on\/ $B\in\CC^{k \times 2 \ell}$.
\end{theorem}

By a similar argument, the polynomial
\begin{equation*}
Q_{B}(\la) = \det(B^TB-\la J_{2\ell})
\end{equation*}
is also invariant under the action \eqref{jointaction}. 
However, it is essentially the same as $P_{B}(\la)$.

\begin{lemma}\label{polyequiv}
With the above notation, we have
\begin{equation}\label{QeqP}
Q_{B}(\la) = (-1)^k\la^{2\ell-k}P_{B}(-\la).
\end{equation}
\end{lemma}

\begin{proof}
Sylvester's determinant identity (see, e.g., \cite{P}) states that if $U$ and $W$ are matrices of size $m \times n$ and $n \times m$ respectively, then 
\begin{equation*}
\det(I_m + UW) = \det(I_n + WU).
\end{equation*}
We replace $U$ with $\la^{-1}A^{-1}U$, where $A$ is an invertible $m \times m$ matrix, thus obtaining
\begin{equation*}
\det(\la A+UW) = \la^{m-n} \det(\la I_n + WA^{-1}U)\det(A).
\end{equation*}
Letting $m=2\ell$, $n=k$,  $A=-J_{2\ell}$, and $U^T=W=B$, we get \eqref{QeqP}.
\end{proof}

Along with the constant terms
\begin{align*}
p_0(B) &= P_{B}(0) = \det (BJ_{2\ell}B^T), \\
q_0(B)&=Q_B(0)= \det (B^TB),
\end{align*}
 the coefficient of $\la^{k-2}$ in $P_B(\la)$ will also be useful in the following section. Let $R_1,\ldots, R_{k}$ be the rows of $B$. Then up to a sign, this coefficient is given by 
\begin{equation}\label{inv2}
\begin{split}
p_{k-2}(B)  &=\sum_{1 \leq s_1 < s_2 \leq k }\bigl{(}R_{s_i}J_{2\ell}R_{s_j}^T \bigr{)}^2\\
&= \sum_{1 \leq s_1 < s_2 \leq k } \biggl{(} \sum_{t=1}^\ell\det \M{cc}{
b_{s_1,2t-1} & b_{s_1,2t} \\
b_{s_2,2t-1} & b_{s_2,2t}
}
\biggr{)}^2.
\end{split}
\end{equation}

\begin{remark}\label{qobs}
The following observation will be useful in the classification given in the next section. Assume the matrix $B$ satisfies $C^TC=0$ for every linear combination $C$ of columns of $B$. Then this same property holds for every matrix in the orbit of $B$ under the action \eqref{jointaction}.
\end{remark}

\section{Classification up to dimension $7$}\label{sec:class}
In this section, for each superspace $V$ with $\dim V\le 7$, we will find a representative of each equivalence class of irreducible pre-oscillator forms on $V$. 
First, let $b = (b_1,\ldots,b_k)^T \in\CC^k$ be a column vector and $q(b)= b^Tb$. We aim to find a canonical form for a representative of the orbit of $b$ under the left action $Xb$ of $X\in\Or(k)$.

\begin{proposition}
Let\/ $b \in \CC^k$ be a column vector. If\/ $q(b) \neq 0$, then under the left action of\/ $\Or(k)$, $b$ is in the orbit of 
\begin{equation}\label{vectqnot0}
(\sqrt{q(b)},0,\ldots,0)^T,
\end{equation}
where\/ $\sqrt{a}$ is defined to be the unique element\/ $\ga$ of 
\begin{equation*}
\CC^+ = \{\ga \in \CC : \re \ga > 0 \textnormal{ or } \re \ga = 0 \textnormal{ and } \im \ga > 0\}
\end{equation*}
satisfying\/ $\ga^2 = a$. On the other hand, if\/ $q(b) = 0$, then either\/ $b=0$ or\/ $b$ is in the orbit of 
\begin{equation}\label{bzero}
(1,\ii,0,\ldots,0)^T.
\end{equation}

\end{proposition} 
\begin{proof}
If $q(b) \neq 0$, it is enough to give the proof in the case when $k=2$, because the general case can be reduced to that. Let $b = (b_1,b_2)^T$ be such that $q(b) \neq 0$. Then the matrix
\begin{equation*}
X = \frac{1}{\sqrt{q(b)}}\M{cc}{
b_1 & b_2 \\
-b_2 & b_1
} \in \Or(2)
\end{equation*}
satisfies $Xb = (\sqrt{q(b)},0)^T$.

Now assume $q(b) = 0$. If $k=2$, then $q(b)=0$ implies $b = (b_1,\pm \ii b_1)^T$. These two possible forms for $b$ are in the same orbit, so without loss of generality $b = (b_1,\ii b_1)^T$. Then $Xb = (1,\ii)^T$ for
\begin{equation}\label{rescale}
X= \frac{1}{2b_1}\M{cc}{
b_1^2+1 & \ii(b_1^2-1) \\
-\ii(b_1^2-1) & b_1^2+1
} \in \Or(2).
\end{equation}
Now suppose $k\geq 3$ and $b$ is nonzero. Then possibly after reordering, we may assume $b_1\neq 0$. Then 
\begin{equation*}
b_2^2+\cdots+b_k^2 = -b_1^2 \neq 0.
\end{equation*} 
There exists an orthogonal transformation $X$ that leaves $b_1$ invariant and replaces $(b_2,\ldots,b_k)^T$ with a vector of the form \eqref{vectqnot0}. Thus we obtain
\begin{equation*}
Xb = (b_1,\ii b_1, 0, \ldots, 0)^T.
\end{equation*}
Then using an orthogonal transformation that acts as \eqref{rescale} on rows 1 and 2 and as the identity on the remaining rows, we obtain \eqref{bzero}.
\end{proof}

Now we consider the right action of the symplectic group: $B\mapsto BY$ for $B\in\CC^{k\times 2\ell}$ and
$Y\in\Sp(2\ell)$.
Let $C_1,\ldots,C_{2\ell}$ be the columns of $B$.   For $1 \leq i \leq \ell$, we will say columns $C_{2i-1},C_{2i}$ are \textit{paired}. 
Using suitable $Y\in\Sp(2\ell)$, we can perform the following elementary operations on paired columns of $B$.
\begin{enumerate}
\item[i.] Rescaling by $\la \neq 0$:
\begin{equation*}
(\ldots, C_{2i-1},C_{2i},\ldots) \mapsto (\ldots, \la C_{2i-1},\la^{-1}C_{2i}, \ldots).
\end{equation*}
\item[ii.] Adding any multiple of a column to its paired column:
\begin{equation*}
(\ldots, C_{2i-1},C_{2i},\ldots) \mapsto (\ldots, C_{2i-1},C_{2i}+\la C_{2i-1}, \ldots).
\end{equation*}
\item[iii.] Switching columns:
\begin{equation*}
(\ldots, C_{2i-1},C_{2i},\ldots) \mapsto (\ldots, C_{2i},  -C_{2i-1}, \ldots).
\end{equation*}
\end{enumerate}
The following are elementary operations outside of pairs.
\begin{enumerate}
\item[iv.] Adding a multiple of a column to a column other than its pair:
\begin{align*}
(\ldots&,C_{2i-1},C_{2i},\ldots,C_{2j-1},C_{2j},\ldots) \\ 
&\mapsto (\ldots,C_{2i-1}-\la C_{2j-1},C_{2i},\ldots,C_{2j-1},C_{2j}+\la C_{2i},\ldots).
\end{align*}
\item[v.] Switching pairs of columns:
\begin{align*}
(\ldots,C_{2i-1},C_{2i},\ldots,&C_{2j-1},C_{2j},\ldots)\\
& \mapsto (\ldots,C_{2j-1},C_{2j},\ldots,C_{2i-1},C_{2i},\ldots).
\end{align*}
\end{enumerate}

From the above discussion, we see that using the orthogonal action, we can always reduce to at most $4\ell$ nonzero rows. Also, using the symplectic action we can always reduce so that at most the first $2k$ columns are nonzero. Thus if $V$ has an irreducible bilinear form, its even and odd dimensions must satisfy $\ell \leq k \leq 4\ell$. It follows that to obtain a complete classification of irreducible pre-oscillator forms up to dimension $7$, we only need to consider the following cases.

\subsection*{Case $k=\ell=1$}
It is easy to see that $B$ is in the orbit of
\begin{equation*}
\B_1 = \M{cc}{
1 & 0
}.
\end{equation*}

\subsection*{Case $k=\ell=2$} If there exists a linear combination $C=\la_1C_1 +\cdots + \la_{2\ell}C_{2\ell}$ of columns satisfying $q(C) \neq 0$, then we can use the symplectic action to put $C$ in the first column of $B$. Then using the orthogonal action and rescaling we obtain
 \begin{equation*}
B= \M{cccc}{
 1 & b_{12} & b_{13} & b_{14} \\
 0 & b_{22} & b_{23} & b_{24} 
}.
 \end{equation*}
Using the symplectic action we eliminate $b_{13}$ and $b_{14}$ followed by $b_{12}$ obtaining a matrix of the form
\begin{equation*}
B= \M{cccc}{
 1 & 0 & 0 & 0 \\
 0 & b_{22} & b_{23} & b_{24} 
}.
\end{equation*}
Then again via the symplectic action we get
\begin{equation*}
B= \M{cccc}{
 1 & 0 & 0  & 0 \\
 0 & \al& 0 & 0
}.
\end{equation*}
If $\al=0$, this reduces to $\B_1$. If $\al \neq 0$, then this reduces to the case 
\begin{equation*}
\B_{2,\al}= \M{cc}{
1 &0\\
0 & \al  
} \qquad (\al \in \CC^+).
\end{equation*}
Finally, assume every linear combination $C$ of the columns of $B$ satisfies $q(C) = 0$. Then using the orthogonal action we obtain
\begin{equation*}
B=\M{cccc}{
1 & b_{12} & b_{13} & b_{14} \\
\ii & \ii b_{12} &\ii b_{13} & \ii b_{14} \\
}.
\end{equation*}
Using the symplectic action we eliminate columns 2, 3, and 4. Thus this reduces to the case
\begin{equation*}
\B_{3} = \M{cc}{
1 &0 \\
\ii & 0 
}.
\end{equation*}
Notice that this argument also shows that every irreducible matrix $B$ of size $k=2,\ell=1$ is in the orbit of $\B_{2,\al}$ or $\B_3$. The details of the remaining cases are similar to those already shown, so we omit them and provide the results.

\subsection*{Case $k=3,\ell=1$} The matrix $B$ is in the orbit of 
\begin{equation}
\B_{4} = \M{cc}{
1 & 0 \\
0 & 1 \\ 
0 & \ii \\
},
\end{equation}
or it can be reduced to $\B_1$, $\B_{2,\al}$, or $\B_3$.

\subsection*{Case $k=3,\ell=2$} Either the matrix $B$ is in the orbit of 
\begin{equation*}
\B_{5} = \M{cccc}{
1 & 0 & 0 & 0 \\
0 & 1 & 1 & 0 \\
0 & 0 & 0 & \ii \\
},
\end{equation*}
or it can be reduced to an orthogonal direct sum involving the previous four irreducible cases.

\subsection*{Case $k=4,\ell=1$}
Either $B$  is in the orbit of 
\begin{equation*}
\B_{6} = \M{cc}{
1 & 0 \\
\ii & 0 \\
0 & 1 \\ 
0 & \ii \\
},
\end{equation*}
or it reduces to one of $\B_1$, $\B_{2,\al}$, $\B_3$, or $\B_4$.


\begin{theorem}
Let\/ $V = \Ve \oplus \Vo$ be a superspace of dimension\/ $\leq 7$. The following is a complete list up to equivalence of inhomogeneous irreducible supersymmetric bilinear forms on\/ $V$ whose restrictions to\/ $\Ve$ and\/ $\Vo$ are 
nondegenerate\/ $(\al \in \CC^+){:}$
\begin{center}
\begin{equation*}
\renewcommand\arraystretch{1.3}
\begin{array}{ccc}
\hline
\dim \Ve & \dim \Vo & B \\
\hline
1 & 2 & \B_1 \\
2 & 2 & \B_{2,\al}, \  \B_3 \\
3 & 2 & \B_4 \\
3 & 4 & \B_5 \\
4 & 2 & \B_6 \\
\hline
\end{array}.
\end{equation*}
\end{center}
\end{theorem}
\begin{proof}
We have shown that every matrix $B$ corresponding to a bilinear form on a superspace of dimension up to seven is either reducible or in the orbit of one of the representatives listed in the table. In order to complete the proof of the theorem we need to check that the entries in the last column of the table are irreducible and distinct, i.e., in different orbits.

The matrix $\B_1$ clearly corresponds to an irreducible form. We note that $\det(\B_3)^2 = 0$ and $\det(\B_{2,\al})^2 = \al^2$. Thus these matrices are all in distinct orbits. If any of these matrices reduces, it must be in the orbit of 
\begin{equation}\label{red2by2}
\M{cc}{
1 & 0 \\
0 & 0 
},
\end{equation}
but this matrix has determinant 0 and therefore cannot share an orbit with $\B_{2,\al}$. So $\B_{2,\al}$ is irreducible for $\al \in \CC$. That $\B_3$ is not in the orbit of \eqref{red2by2} follows from Remark \ref{qobs}. 

The matrix $\B_4$ has rank 2 and so reduces only if it is in the orbit of the matrix
\begin{equation}\label{orbit3}
\M{cc}{
1 & 0 \\
0 & \al \\
0 & 0
}
\end{equation}
for some $\al \in \CC^+$. But the invariant $p_1$ (cf.\ (\ref{inv2})) evaluated on the matrix \eqref{orbit3}  is $\al^2 \neq 0$, whereas $p_1(\B_{4}) = 0$. So $\B_4$ is irreducible. The matrix $\B_6$ has rank 2 and is therefore reducible only if it is in the orbit of $\B_4$ or the matrix
\begin{equation}\label{orbit4}
\M{cc}{
1 & 0 \\
0 & \al \\
0 & 0\\
0 & 0
}
\end{equation}
for some $\al \in \CC^+$. By Remark \ref{qobs}, $\B_6$ is not in the orbit of $\B_4$ or \eqref{orbit4}. So $\B_6$ is irreducible.

$\B_5$ has rank three and so is reducible only if it is in the orbit of the matrix 
\begin{equation}\label{orbit5}
\M{cccc}{
1 & 0  & 0 & 0 \\
0 & \al & 0 & 0 \\
0 & 0 & 1 & 0\\
}
\end{equation}
for some $\al \in \CC^+$. But the invariant $p_1$ shows that  $\B_5$ is not in the same orbit as \eqref{orbit5}. So $\B_5$ is irreducible.
\end{proof}

Applying our methods to higher dimensions, we have found that the classification seems to become increasingly complicated as $\dim V$ increases. Though it should be possible to extend the classification to dimension $8$ or $9$ with the methods of this paper, a more general approach will be needed for a complete classification in any dimension.
We plan to address this question in the future by using the theory of \emph{$\theta$-groups} (see \cite{DK, K3}).


\section{Oscillator Lie superalgebras}\label{sec:oscillator}
As we explained in the introduction, oscillator Lie superalgebras can be obtained from skew-supersymmetric bilinear forms. Now we show how one obtains an inhomogoneous oscillator superalgebra from an inhomogeneous pre-oscillator form. Then we use an example of such a superalgebra to construct a trivial abelian extension of $\lie{osp}(1|2)$.

Let $V=\Ve \oplus \Vo$ be a superspace and $(\cdot|\cdot)\colon V\times V \to \CC$ be a skew-supersymmetric pre-oscillator form. Consider the extension 
\begin{equation}\label{2dimext}
A= V \oplus \CC K \oplus \CC \ka.
\end{equation}
 We give this extension the structure of a Lie superalgebra by declaring $K$ and $\ka$ to be even and odd central elements respectively and letting
\begin{equation}
[a,b] = \begin{cases}
(a|b)K, & p(a)=p(b) \\
(a|b)\ka, & p(a) \neq p(b)
\end{cases}
\end{equation}
for all homogeneous vectors $a,b \in V$.

Before considering an example of such a form, let us recall a construction of $\lie{osp}(1|2)$ as a subalgebra of a homogeneous oscillator algebra. Assume $k=\ell = 1$ and $(\cdot|\cdot)$ is a nondegenerate, even, skew-supersymmetric bilinear form. Then there exist bases $\{b_1,b_2\}$ of $\Ve$ and $\{a\}$ of $\Vo$ such that the Gram matrix of the bilinear form is 
\begin{equation*}
G= \M{cc|c}{
0 & 1 & 0 \\
-1 & 0 & 0 \\
\hline
0 & 0 & 1
}.
\end{equation*}
The 
 central extension $A=V \oplus \CC K$ has nonzero brackets
\begin{equation*}
[b_1,b_2] = K, \qquad \{a,a\} = K,
\end{equation*}
where we use $\{\cdot,\cdot\}$ to denote the superbracket of odd vectors.
Then $\lie{osp}(1|2)$ is realized as a subalgebra of the oscillator algebra $\mathcal{U}(A)$ as follows
(see, e.g., \cite{FAP}):
\begin{align}
H&=\frac{1}{4}b_2b_1+\frac{1}{4}b_1b_2 , & F^+&=\frac{1}{4}\ ab_2+\frac{1}{4}b_2a,\nonumber\\
E^+&=\frac{1}{2}b_2^2, & F^-& = \frac{1}{4}\ ab_1+ \frac{1}{4}b_1a,\label{osp12}\\
E^-&=-\frac{1}{2}b_1^2,& \nonumber
\end{align}
with brackets
\begin{align*}
[H,E^{\pm}]&=\pm E^{\pm}, & [E^+,E^-]&=2H, \nonumber\\
[H,F^{\pm}]&=\pm\frac{1}{2}F^{\pm}, & \{F^+,F^-\}&=\frac{1}{2}H, \\
[E^{\pm},F^{\mp}]&=-F^{\pm}, & \{F^{\pm},F^{\pm}\} &= \pm\frac{1}{2}E^{\pm}\nonumber.
\end{align*}
We obtain a highest weight representation of $\mathcal{U}(A)$ on a Fock space $\mathcal{F} = \CC[x] \otimes \CC \xi$ where $\xi$ is an odd indeterminate satisfying $\xi^2 = \frac{1}{2}$. The action of $A$ on $\mathcal{F}$ is given by
\begin{equation*}
b_1 \mapsto \partial_x, \qquad b_2 \mapsto x, \qquad  a \mapsto \xi, \qquad K \mapsto \Id.
\end{equation*}

Now consider the case when the Gram matrix is given instead by 
\begin{equation*}
G=\M{cc|c}{
0 & 1 & 1 \\
-1 & 0 & 0 \\
\hline
-1 & 0 & 1\\
}.
\end{equation*}
This inhomogeneous form gives rise to the central extension \eqref{2dimext} with nonzero brackets 
\begin{equation*}
[b_1,b_2] = K,\qquad  [b_1,a]=\ka, \qquad \{a,a\} = K.
\end{equation*}
We consider the same elements \eqref{osp12} of $\mathcal{U}(A)$ we used to construct $\lie{osp}(1|2)$ in the previous example. Then the brackets become
\begin{align}\label{extbrackets}
[H,E^{\pm}]&=\pm E^{\pm}, &\nonumber [E^-,F^-] &= \ka E^-,\\
[H,F^+]&=\frac{1}{2}F^++\frac{\ka}{2} E^+, &  [E^+,E^-]&=2H,  \nonumber \\
[H,F^-]&=-\frac{1}{2}F^-+\frac{\ka}{2} H, & \{F^+,F^-\}&=\frac{1}{2}H -\frac{\ka}{2} F^+,\\
[E^{+},F^-]&=-F^+, &  \nonumber \{F^+,F^+\}&=\frac{1}{2}E^+, \\
[E^-,F^+]&=-F^--\ka H,  & \{F^-,F^-\} &= -\frac{1}{2}E^--\ka F^-.\nonumber
\end{align}

Observe that $\ka^2=0$ in $\mathcal{U}(A)$.
Thus the brackets of the $\lie{osp}(1|2)$ subalgebra have been modified by elements of the abelian ideal 
\begin{equation*}
M=\Span\{\ka H, \ka E^{\pm},\ka F^{\pm}\}.
\end{equation*}
We note that $\Span\{H,E^{\pm},F^{\pm}\}$ acts on $M$ as the adjoint representation of $\lie{osp}(1|2)$, and thus we have a Lie superalgebra structure on the vector space $L = \lie{osp}(1|2) \oplus M$ such that the projection $\pi \colon L \to \lie{osp}(1|2)$ is a surjective homomorphism, and the restriction of the adjoint representation of $L$ to $M$ yields the original action of $\lie{osp}(1|2)$ on $M$. Thus $L$ is an abelian extension of $\lie{osp}(1|2)$ by its adjoint representation viewed as an abelian Lie superalgebra with parities reversed.

Denote by $[\cdot, \cdot]_L$ the bracket on $L$ given by (\ref{extbrackets}) and let 
\begin{equation*}
\ga \colon \lie{osp}(1|2) \times\lie{osp}(1|2)\to  \lie{osp}(1|2)
\end{equation*}
satisfy
\begin{equation*}
[a,b]_L = [a,b] + \ka \ga(a,b).
\end{equation*}
This extension is trivial if there exists an odd linear map $f \colon \lie{osp}(1|2) \to \lie{osp}(1|2)$ such that 
\begin{equation}\label{trivial}
\ga(a,b) = (-1)^{p(a)} [a,f(b)] - (-1)^{(p(a)+1)p(b)}[b,f(a)]-f([a,b])
\end{equation}
for all $a,b \in \lie{osp}(1|2)$. It is straightforward to check that the following choice of $f$ satisfies \eqref{trivial}:
\begin{equation*}
f(H)=f(E^\pm) = 0, \qquad f(F^+) = E^+, \qquad f(F^-) = H.
\end{equation*}

As before, we can represent $A$ on the Fock space $\mathcal{F} = \CC[x]\otimes \bigwedge(\xi,\ka)$ 
where $\xi$ is odd with $\xi^2 = \frac{1}{2}$ and $\ka$ acts as an odd indeterminate that we denote again by $\ka$. Then the action of $A$ on $\mathcal{F}$ is given by
\begin{equation*}
b_1 \mapsto \partial_x+\ka \partial_{\xi}, \qquad
b_2  \mapsto x,  \qquad
a  \mapsto \xi, \qquad
K \mapsto \Id, \qquad
\ka \mapsto \ka.
\end{equation*}

\section*{Acknowledgements}
We are grateful to Dimitar Grantcharov and Victor Kac for many valuable discussions. We would like to thank the referee for carefully reading the manuscript and suggesting improvements.

\bibliographystyle{amsalpha}

\end{document}